\documentclass[11pt]{amsart}

\usepackage{amsmath}
\usepackage{amsfonts}
\usepackage{amssymb}
\usepackage{amsthm}
\usepackage{mathrsfs}
\usepackage{latexsym}
\usepackage{times}
\usepackage[pagebackref]{hyperref}
\usepackage[backrefs,alphabetic,initials]{amsrefs}

\numberwithin{equation}{section}
\usepackage{url}
\usepackage{graphicx}

\newtheorem{thm}{Theorem}[section]
\newtheorem{lem}[thm]{Lemma}
\newtheorem{prop}[thm]{Proposition}
\newtheorem{cor}[thm]{Corollary}

\theoremstyle{definition}
\newtheorem{defn}[thm]{Definition}
\newtheorem*{aknow}{Acknowledgments}

\theoremstyle{remark}
\newtheorem{rem}[thm]{Remark}

\newcommand{\A}{\mathcal{A}}
\newcommand{\B}{\mathcal{B}}
\newcommand{\G}{\mathcal{G}}
\newcommand{\R}{\mathbb{R}}

\newcommand{\canbangg}[1]{ \begin{array}{l} #1 \end{array} }

\begin{document}

\title[Lie Groups Whose Coadjoint Orbits Have Only Zero or Maximal Dimension]{Classification of Real Solvable Lie Algebras Whose Simply Connected Lie Groups Have Only Zero or Maximal Dimensional Coadjoint Orbits}

\author[Anh Vu Le et al.]{Anh Vu Le \and Van Hieu Ha \and Anh Tuan Nguyen \and \\ Tran Tu Hai Cao \and Thi Mong Tuyen Nguyen}

\address{Anh Vu Le, Department of Economic Mathematics, University of Economics and Law, Vietnam National University - Ho Chi Minh City, Viet Nam.}
\email{vula@uel.edu.vn}

\address{Van Hieu Ha, Department of Economic Mathematics, University of Economics and Law, Vietnam National University - Ho Chi Minh City, Viet Nam.}
\email{hieuhv@uel.edu.vn}

\address{Anh Tuan Nguyen, Faculty of Political Science and Pedagogy, University of Physical Education and Sports, Ho Chi Minh City, Viet Nam.}
\email{natuan@upes.edu.vn}

\address{Tran Tu Hai Cao, Le Quy Don High School for the Gifted, Ninh Thuan Province, Viet Nam.}
\email{tuhai.thptlequydon@ninhthuan.edu.vn}

\address{Thi Mong Tuyen Nguyen, Faculty of Mathematics and Information, Dong Thap University, Cao Lanh city, Dong Thap Province, Viet Nam.}
\email{ntmtuyen@dthu.edu.vn}



\date{\today}

\begin{abstract}
    In this paper we study a special subclass of real solvable Lie algebras having small dimensional or small codimensional derived ideal. It is well-known that the derived ideal of any Heisenberg Lie algebra is 1-dimensional and the derived ideal of the 4-dimensional real Diamond algebra is 1-codimensional. Moreover, all the coadjoint orbits of any Heisenberg Lie group as well as 4-dimensional real Diamond group are orbits of dimension zero or maximal dimension. In general, a (finite dimensional) real solvable Lie group is called an $MD$-group if its coadjoint orbits are zero-dimensional or maximal dimensional. The Lie algebra of an $MD$-group is called an $MD$-algebra and the class of all $MD$-algebras is called $MD$-class. Simulating the mentioned above characteristic of Heisenberg Lie algebras and 4-dimensional real Diamond algebra, we give a complete classification of $MD$-algebras having 1-dimensional or 1-codimensional derived ideals.
\end{abstract}

\maketitle

\section{Classification of solvable Lie algebras: a quick introduction}

Classifying all Lie algebras of dimension less than 4 is an elementary exercise. However, when considering Lie algebras of dimension $n \, (n \geqslant 4)$, complete classifications are much harder. As it has long been well known, there exist three different types of Lie algebras: the semisimple, the solvable and those which are neither semi-simple nor solvable. By the Levi-Maltsev Theorem \cite{Mal45} in 1945, any finite-dimensional Lie algebra over a field of characteristic zero can be expressed as a semidirect sum of a semi-simple subalgebra and its maximal solvable ideal. It reduces the task of classifying all finite-dimensional Lie algebras to obtaining the classification of semi-simple and of solvable Lie algebras.
 
The problem of the classification of semi-simple Lie algebras over the complex field has been completely classified by Killing, E. Cartan \cite{Car94} in 1894, over the real field by F. R. Gantmakher \cite{Gan39} in 1939.

Although several classifications of solvable Lie algebras of small dimension are known, but the problem of the complete classification of the (real or complex) solvable Lie algebras is still open up to now. There are two ways of proceeding in the classification of solvable Lie algebras: {\bf by dimension} or {\bf by structure}.

First, we list some results about the classification of solvable Lie algebras in the dimensional approach.

\begin{itemize}
	\item All solvable Lie algebras up to dimension 6 over the complex field $\mathbb{C}$ and the real field $\mathbb{R}$ were classified by G. M. Mubarakzyanov \cite{Mub63} in 1963 and by P. Turkowski \cite{Tur90} in 1990.  
	\item All solvable Lie algebras up to dimension 4 over any perfect field were classified by J. Patera and H. Zassenhaus \cite{PZ90} in 1990. 
	\item Some incomplete classifications of solvable Lie algebras in dimension 7 and nilpotent algebras up to dimension 8 were given by G. Tsagas \cite{Tsa99} in 1999. 
\end{itemize}

It seems to be very difficult to proceed by dimension in the classification of Lie algebras of dimension greater than 6. However, it is possible to proceed by structure, i.e. to classify solvable Lie algebras with a specific given property. Now, we list some results about the classification of solvable Lie algebras in the structural approach.

\begin{itemize}
	\item In 1973, M. A. Gauger \cite{Gau73} gave a complete classification of metabelian Lie algebras of dimension no more than 7 and nearly complete results for dimension 8.
	\item In 1995, D. Arnal, M. Cahen and J. Ludwig \cite{ACL95} gave the list of all solvable Lie algebras such that the coadjoint orbits of the connected Lie groups corresponding to them are of dimension zero or two. But they have not classified them yet, up to isomorphism.
	\item In 1999, L. Yu. Galitski and D. A. Timashev \cite{GT99} completely classified all of metabelian Lie algebras of dimension 9.
	\item In 2007, R. Campoamor-Stursberg \cite{CS07} gave a complete classification of nine-dimensional Lie algebras with nontrivial Levi decomposition.
	\item In 2007, I. Kath \cite{Kat07} classified the class of nilpotent quadratic Lie algebras of dimension no more than 10.
	\item In 2010, another class of Lie algebras relating to the nilradicals has been being classified by L. \u Snobl \cite{Sno10}
	\item In 2012, M. T. Duong, G. Pinczon, and R. Ushirobira \cite{DPU12} gave a classification of Solvable singular quadratic Lie algebras.  
	\item In 2012, L. Chen \cite{Che12} classified a class of solvable Lie algebras with triangular decompositions.
\end{itemize}
 
In an attempt to classify solvable Lie algebras by structure, we study in this paper a special subclass of real solvable Lie algebras having small dimensional or small codimensional derived ideals. This idea comes from an investigation of Kirillov's Orbit Method on the $(2m+1)$-dimensional Heisenberg Lie algebras ($0 < m \in \mathbb{N}$) and the 4-dimensional real Diamond Lie algebra. Recall that, in 1962, A. A. Kirillov \cite{Kir76} introduced the Orbit Method which quickly became the most important method in the theory of representations of Lie groups and Lie algebras. The key of Kirillov's Orbit Method is the coadjoint orbits or $K$-orbits (i.e., orbits in the coadjoint representation) of Lie groups. We emphasize that any $K$-orbit of the $(2m+1)$-dimensional Heisenberg Lie group and the 4-dimensional real Diamond Lie group has dimension zero or maximal. Hence, it is reasonable to consider the class of solvable Lie groups (and corresponding algebras) having the similar property. A (finite dimensional) real solvable Lie group is called an $MD$-group (in term of N. D. Do \cite{Do99}) if its $K$-orbits are orbits of dimension zero or maximal dimension. The Lie algebra of an $MD$-group is called an $MD$-algebra and the class of all $MD$-algebras is called $MD$-class. In particular, if the maximal dimension of the $K$-orbits of some $MD$-group $G$ is equal to $\dim G$ then $G$ is called an $SMD$-group and its algebra is called an $SMD$-algebra. The class of all $SMD$-algebras is called $SMD$-class. It is clear that $SMD$-class is a subclass of $MD$-class.

The investigation of $MD$-class was first time suggested by N. D. Do \cite{Do99} in 1982. Now, we list main results about $MD$-class.

\begin{itemize}
	\item In 1984, H. V. Ho \cite{VH84} completely classified all of $SMD$-algebras (of arbitrary dimension).
	\item In 1990, A. V. Le \cites{Le90-1, Le90-2, Le93} gave a complete classification of  all 4-dimensional $MD$-algebras.
	\item In 1995, D. Arnal, M. Cahen and J. Ludwig \cite{ACL95} gave the list of all $MD$-algebras such that the maximal dimension of $K$-orbits of corresponding $MD$-groups is just 2, but they have not yet classified them up to isomorphism.
	\item Up to 2012, A. V. Le et al. \cites{LS08, LHT11} had classified (up to isomorphism) all of $MD$-algebras of dimension 5. 
	\item In 2013, the $MD$-class was listed as a specific attention in classification of Lie Algebras by L. Boza, E. M. Fedriani, J. Nunez and A. F. Tenorio \cite{BFNT13}.
\end{itemize}

The investigation of general properties of $MD$-class, in particular, the complete classification of $MD$-class is still open up to now. 

As we say above, the $(2m+1)$-dimensional real Heisenberg Lie algebra and the 4-dimensional real Diamond Lie algebra are $MD$-algebras. The real Lie Heisenberg algebras and their extensions are investigated by a lot of mathematicians because of their physical origin and applications. Moreover, the first derived ideal of the Heisenberg Lie algebra is 1-dimensional and the first derived ideal of the 4-dimensional Diamond Lie algebra is 1-codimensional. We will generalize these properties to consider $MD$-algebras having the first derived ideal of dimension 1 or codimension 1. For convenience, we shall denote by $MD(*,1)$ or $MD(*,*-1)$ the subclasses of $MD$-algebras having 1-dimensional or 1-codimensional derived ideals, respectively. If $\G$ belongs to $MD(*,1)$ or $MD(*,*-1)$ then it is called an $MD(*,1)$-algebra or $MD(*,*-1)$-algebra, respectively. In particular, every $MD(*,1)$-algebra or $MD(*,*-1)$-algebra of dimension $n$ is called an $MD(n,1)$-algebra or $MD(n,n-1)$-algebra, respectively. Of course, the $(2m+1)$-dimensional Heisenberg Lie algebra belongs to $MD(2m+1,1)$ and the 4-dimensional real Diamond Lie algebra belongs to $MD(4,3)$. The main purpose of this paper is to completely classify, up to isomorphism, $MD(*,1)$-class and $MD(*,*-1)$-class. We also prove that any real solvable Lie algebra having 1-dimensional derived ideal belongs to $MD(*,1)$ and give a sufficient and necessary condition in order that a $n$-dimensional real solvable Lie algebra having 1-codimensional derived ideal belongs to $MD(n,n-1)$ with $n > 4$.

    The next part of the paper will be organized as follows: Section 2 gives some basic concepts, especially we recall the Lie algebra of the group of the affine transformations of real straight line, the real Heisenberg Lie algebras and the real Diamond Lie algebras. Section 3 deals with some well-known remarkable classifications of some subclasses of $MD$-algebras. The main results about the complete classifications of $MD(*, 1)$-class and $MD(*, * - 1)$-class, are given in Section 4. The last section is devoted the discussion of some open problems.

\section{Some basic concepts}
     
     We first recall in this section some preliminary results and notations which will be used later. For details we refer the reader to the book \cite{Kir76} of A. A. Kirillov and the book \cite{Do99} of N. D. Do.
      
\subsection{The coadjoint representation and coadjoint orbits}
      Let $G$ be a Lie group, $\G$ = Lie($G$)  be the corresponding Lie algebra of $G$ and $\G^*$ be the dual space of $\G$. For every $g \in G$, we denote the internal automorphism associated with \textit{g} by $A_{(g)}$, and hence, $A_{(g)}:G \to G$ can be defined as follows 
      $A_{(g)}:=g.x.g^{-1}, \forall x \in G.$
      
      This automorphism induces the following map $A_{\left( g \right)^*} : \G \to \G$  which is defined as follows $$ A_{\left( g \right) ^ *  }\left( X \right): = \frac{d}
      {{dt}}\left[ {g.\exp \left( {tX} \right).g^{ - 1} } \right]\left| {_{t = 0} } \right.;\, \forall X 
      \in {\mathcal{G}}.$$
This map is called \emph{tangent map} of $A_{(g)}$. 

\begin{defn} 
	The action $$\begin{array}{l} Ad:G \to Aut\left(\G \right)\\
	\,\,\,\,\,\,\,\,\,\,\,\, \, g \mapsto {A_{{\left( g \right)}^*}} \end{array}$$
	is called \emph{the adjoint representation of $G$ in $\G$}.
\end{defn}

The coadjoint representation is the dual of the adjoint representation. Namely, we have the following definition.

\begin{defn} 
	\emph{The coadjoint representation} or \emph{$K$-representation} $$\begin{array}{l}
	K:G \to Aut\left(\G^* \right)\\
	\,\,\,\,\,\,\,\,\,\,\,\, \, g \mapsto K_{\left(g \right)}
	\end{array}$$
	of $G$ in $\G^*$ is defined by
	$$\left\langle K_{\left(g \right)} F, X \right\rangle : = \left\langle F, Ad\left(g^{-1}
	\right)X \right\rangle ;\left( F \in \G^*, X \in \G \right),$$
	where $\langle F, Y \rangle $ denotes the value of a linear functional $F$ on an arbitrary vector $Y \in \G$.
\end{defn}

A geometrical interpretation of the coadjoint representation of $G$ is as the action by left-translation on the space of right-invariant 1-form on $G$.
 
\begin{defn}
	Each orbit of the coadjoint representation of $G$ is called a \emph{$K$-orbit of $G$}.
\end{defn}

We denote the $K$-orbit containing $F$ by $\Omega_F$. For every $F \in \G^*$, the $K$-orbit containing $F$ can be defined by $\Omega_F := \left\{ {K_{\left( g \right)}F\, \vert\, g \in G}\right\}$. The dimension of every $K$-orbit of an arbitrary Lie group $G$ is always even. In order to define the dimension of the $K$-orbits $\Omega_F$ for each $F$ from the dual space $\G^*$ of the Lie algebra $\G = {\rm Lie} (G)$, it is useful to consider the following (skew-symmetric bilinear) Kirillov form $B_F$ on $\G$ corresponding to $F$: $B_F \left( {X,Y} \right)= \left\langle F,\left[ {X,Y} \right] \right\rangle$ for all $X,Y \in \G$. Denote the stabilizer of $F$ under the co-adjoint representation of $G$ in $\G^*$  by $G_F$ and $\G_F := {\rm Lie} (G_F)$.

We shall need in the sequel of the following result.

\begin{prop}[{\bf see \cite[Section 15.1]{Kir76}}]\label{prop2.4}
	$$\ker B_F  = \G_F and \dim \Omega_F  = \dim \G - \dim \G_F = {\rm rank} B_F.$$
\end{prop}

\subsection{$MD$-groups and $MD$-algebras and some their properties}

\begin{defn} 
	An \emph{$n$-dimensional $MD$-group} or, for brevity, an \emph{$MDn$-group} is an $n$-dimensional real solvable Lie group such that its $K$-orbits are orbits of dimension zero or maximal dimension. The Lie algebra of an $MDn$-group is called an \emph{$MDn$-algebra}. \emph{$MD$-class} and \emph{$MDn$-class} are the sets of all $MD$-algebras (of arbitrary dimension) and $MDn$-algebras, respectively.
\end{defn}

\begin{defn}
	An \emph{$MD(n, m)$-algebra} is an $MDn$-algebra whose the first derived ideal is $m$-dimensional with $m,n \in \mathbb{N}$ and $0 < m < n$. \emph{$MD(n, m)$-class} is the set of all $MD(n, m)$-algebras. In particular, \emph{$MD(*, 1)$-class} and \emph{$MD(*, *-1)$-class} are the sets of all $MD$-algebras (of arbitrary dimension) having the first derived ideal of dimension 1 and codimension 1, respectively.
\end{defn}

\begin{rem}
	Note that all the Lie algebras of dimension $n \, (n \leqslant 3)$ are $MD$-algebras, and moreover they can be listed easily. So we only take interest in $MDn$-algebras for $n \geqslant 4$.
\end{rem}

For any real Lie algebra $\G$, as usual, we denote the first and second derived ideals of $\G$ by $\G^1: = [\G, \G]$ and $\G^2: = [\G^1, \G^1]$, respectively. Now, we introduce some well-known properties of $MD$-algebras.

First, the following proposition gives a necessary condition for a Lie algebra belonging to $MD$-class.

\begin{prop}[{\bf see \cite[Theorem 4]{VH84}}]\label{prop2.8}
	Let $\G$  be an $MD$-algebra. Then its second derived ideal $\G^2$ is commutative.
\end{prop}

We point out here that the converse of the above result is in general not true. In other words, the above necessary condition is not a sufficient condition.

\begin{prop}[{\bf see \cite[Chapter 2, Proposition 2.1]{Do99}}]
	 Let $\G$ be an $MD$-algebra. If $F \in \G^*$ is not vanishing perfectly in $\G^1$, i.e. there exists $U \in \G^1$ such that $\langle F, U \rangle \neq 0$, then the $K$-orbit $\Omega_F$ has maximal dimension.
\end{prop}

\begin{prop}[{\bf see \cite{LHT11}}]\label{prop2.10}
	There is no $MD$-algebra\, $\G$ \,such that its second derived ideal \,$\G^2$ \,is not trivial and \,$\dim \G^2 = \dim \G^1 - 1$. In other words, if \,$0 < \dim \G^2 = \dim \G^1 - 1$ \,then \,$\G$\, is not an $MD$-algebra.
\end{prop}

To illustrate and show the role of the $MD$-class, in the rest of this section, we will introduce some typical examples and counter-examples of $MD$-algebras.

\subsection{The Lie algebra of the group of affine transformations of real straight line}

The Lie algebra ${\rm aff} (\R)$ of the group ${\rm Aff} (\R)$ of affine transformations of real straight line $\R$ is the unique non-commutative real Lie algebra of dimension 2 and it is defined as follows:
	\[ {\rm aff} (\R) := Span( X, Y ); \,\, [X,Y] = Y.\]

\begin{rem}\label{rem2.11}
	Clearly, every real Lie algebra of dimension $n \leqslant 3$ is an $MD$-algebra. In particular ${\rm aff} (\R)$ is an $MD(2,1)$-algebra.
\end{rem}

\subsection{The real Heisenberg Lie algebras} 

The $(2m + 1)$-dimensional real Heisenberg Lie algebra ($0 < m \in \mathbb{N}$) is the following real Lie algebra:
$$\mathfrak{h}_{2m + 1}: = Span\left( X_i, Y_i, Z \, \vert \, i = 1, 2, \ldots, m\right) ; \, \, [X_i, Y_i] = Z; \, \, 
i = 1, 2, \ldots , m; $$
the other Lie brackets are trivial.

\begin{rem}
The first derived ideal $\mathfrak{h}^1_{2m + 1} = \mathbb{R}.Z = Span(Z)$ of $\mathfrak{h}_{2m + 1}$ is 1-dimensional and coincides with the center of $\mathfrak{h}_{2m + 1}$.
\end{rem}

Let $\left( X_1^*, Y_1^*, \ldots , X_m^*, Y_m^*, Z^* \right)$ be the dual basis of $\left( X_1, Y_1, \ldots , X_m, Y_m, Z \right)$ in the dual space $\mathfrak{h}^*_{2m + 1}$ of $\mathfrak h_{2m+1}$, and
	\[F =  a_1 X_1^* + b_1 Y_1^* + \ldots + a_m X_m^* + b_m Y_m^* + cZ \equiv (a_1, b_1, \ldots, a_m, b_m, c)\]
be an arbitrary element in $\mathfrak{h}^*_{2m + 1}$. Then the Kirillov form $B_F$ is given by the following matrix
	\[
		B_F =  \begin{bmatrix}
		            0 & c & 0 & 0 & \cdots & 0 & 0 & 0 \\
					-c& 0 & 0 & 0 & \cdots & 0 & 0 & 0 \\
					0 & 0 & 0 & c & \cdots & 0 & 0 & 0 \\
					0 & 0 &-c & 0 & \cdots & 0 & 0 & 0 \\
					\vdots & \vdots & \vdots & \vdots & \ddots & \vdots & \vdots & \vdots \\
					0 & 0 & 0 & 0 & \cdots & 0 & c & 0 \\
					0 & 0 & 0 & 0 & \cdots &-c & 0 & 0 \\
					0 & 0 & 0 & 0 & \cdots & 0 & 0 & 0 \\
					\end{bmatrix}
					= {\rm diag} \, (\Lambda, \cdots, \Lambda, 0)
	\]
with $m$ blocks $\Lambda = \begin{bmatrix} 0 & c \\ - c & 0 \end{bmatrix}$.

In view of Proposition \ref{prop2.4}, it is a simple matter to get the following proposition.

\begin{prop}
$\mathfrak{h}_{2m + 1}$ is one $MD(2m+1, 1)$-algebra and the maximal dimension of $K$-orbits in $\mathfrak{h}^*_{2m + 1}$ is $2m$. Moreover, for $F = (a_1, b_1, \ldots , a_m, b_m, c) \in \mathfrak{h}^*_{2m + 1}$, we have
\begin{itemize}
	\item[(i)] $K$-orbits containing $F$ is of dimension $0$ if and only if $c = 0$.
	\item[(ii)] $K$-orbits containing $F$ is of dimension $2m$ if and only if $c \neq 0$.
\end{itemize}
\end{prop}

\subsection{The real Diamond Lie algebras}

The $(2m + 2)$-dimensional real Diamond Lie algebra $(0 < m \in \mathbb{N})$ is one semi-direct extension of the $(2m+1)$-dimensional Heisenberg algebra by $\R$, namely it is the following real Lie algebra:
$$\R.\mathfrak{h}_{2m + 1}: = Span(X_i, Y_i, Z, T \, \vert \, i = 1, 2, \ldots, m)$$
where the Lie structure is given by 
$$[X_i, Y_i] = Z, \, [T, X_i] = - X_i, \, [T, Y_i] = Y_i; \, i = 1, 2, \ldots, m; $$
the other Lie brackets are trivial.

\begin{rem}
	The $(2m+1)$-dimensional real Heisenberg algebra is the first derived ideal of the $(2m + 2)$-dimensional real Diamond Lie algebra. In particular, the first derived ideal of $\R.\mathfrak{h}_{2m + 1}$ is of codimension 1.
\end{rem}

Let \,$\left( X_1^*, Y_1^*, \ldots, X_m^*, Y_m^*, Z^*, T^* \right)$\, be the dual basis of \,$(X_1, Y_1, \ldots, X_m, Y_m,$\\ $Z, T)$ in the dual space $(\R.\mathfrak{h}_{2m + 1})^*$ of $\R.\mathfrak h_{2m+1}$ and $F =  a_1 X_1^* + b_1 Y_1^* + \ldots + +  a_m X_m^* + b_m Y_m^* + cZ + dT \equiv (a_1, b_1, \ldots, a_m, b_m, c, d)$ be an arbitrary element in $(\R.\mathfrak{h}_{2m + 1})^*$. Then we get the Kirillov form $B_F$ as follows
		$$B_F = \begin{bmatrix}
		            0 & c & 0 & 0 & \cdots & 0 & 0 & 0 & a_1 \\
					-c& 0 & 0 & 0 & \cdots & 0 & 0 & 0 & -b_1 \\
					0 & 0 & 0 & c & \cdots & 0 & 0 & 0 & a_2 \\
					0 & 0 &-c & 0 & \cdots & 0 & 0 & 0 & -b_2 \\
					\vdots & \vdots & \vdots & \vdots & \ddots & \vdots & \vdots & \vdots & \vdots \\
					0 & 0 & 0 & 0 & \cdots & 0 & c & 0 & a_m \\
					0 & 0 & 0 & 0 & \cdots &-c & 0 & 0 & -b_m\\
					0 & 0 & 0 & 0 & \cdots & 0 & 0 & 0 & 0 \\
					-a_1 & b_1 & -a_2 & b_2 & \cdots & -a_m & b_m & 0 & 0
					\end{bmatrix}.$$
By virtue of Proposition \ref{prop2.4}, one can verify the following proposition.

\begin{prop}
	The $(2m+2)$-dimensional real Diamond Lie algebra $\R.\mathfrak{h}_{2m+1}$ is an $MD(2m+2, 2m+1)$-algebra if and only if \,$m = 1$. That means the 4-dimensional real Diamond Lie algebra is an $MD(4,3)$-algebra and the $(2m+2)$-dimensional real Diamond Lie algebra is not an $MD$-algebra for every natural number $m > 1$.
\end{prop}


\section{Some subclasses of $MD$-class}\label{Sect3}

In this section, we would like to introduce some well-known remarkable results of classification of $MD$-class. First, recall that all of the $MD$-algebras of dimension 4 or 5 were classified, up to isomorphism, by A. V. Le et al. \cites{Le90-1, Le93, LS08}. However, to illustrate the general results which will be given in the last section of the paper, we will introduce here the classification of $MD(n,1)$-class and $MD(n,n-1)$-class for small $n$, namely $n = 4$ or $n = 5$.


\subsection{Classification of $MD(4,1)$-class and $MD(4,3)$-class}

\begin{prop}[{\bf Classification of $MD(4,1)$-algebras, see \cite{Le93}}]\label{prop3.1}
Let $\G$ be an $MD(4,1)$-algebra. Then $\G$ is decomposable and we can choose a suitable basis $(X, Y, Z, T)$ of $\G$ such that $\G^1 = Span(Z) = \R.Z,$ and $\G$ is isomorphic to one of the following Lie algebras. 
\begin{itemize}
   		\item[1.1.] $\G_{4,1,1}: = \mathfrak{h}_3 \oplus \R.T,\, [X,Y] = Z;$\, the others Lie brackets are trivial.
   		\item[1.2.] $\G_{4,1,2}: = \text{aff} (\R) \oplus \R.Z \oplus \R.T,\, [X,Y] = Y;$\, the others Lie brackets are trivial.
\end{itemize}
\end{prop}

\begin{prop}[{\bf Classification of $\mathbf{\mathit{MD(4,3)}}$-algebras, see \cite{Le93}}]\label{prop3.2}
	Let $\G$ be an $MD(4,3)$-algebra. Then $\G$ must be indecomposable and we can choose a suitable basis $(X, Y, Z, T)$ of $\G$ such that $\G$ is isomorphic to one of the following Lie algebras. 
\begin{itemize}
   \item[1.] $\G^1 = Span(X,Y,Z)\equiv \R^3$, \, $ad_T \in Aut_{\R} \left(\G^1\right) \equiv GL_3(\R)$
       \begin{itemize}
        \item[1.1.] $\G_{4,3,1(\lambda_1, \lambda_2)}:$\,
$ad_T = \begin{pmatrix} \lambda_1 & 0 & 0 \\ 0 & \lambda_2 & 0 \\ 0 & 0 & 1 \end{pmatrix};
\, \lambda_1, \lambda_2 \in \R \setminus \lbrace 0\rbrace.$

        \item[1.2.] $\G_{4,3,2(\lambda)}:$\,
$ad_T = \begin{pmatrix} \lambda & 1 & 0 \\ 0 & \lambda & 0 \\ 0 & 0 & 1 \end{pmatrix};
\, \lambda \in \R \setminus \lbrace 0\rbrace.$

        \item[1.3.] $\G_{4,3,3}:$\,
$ad_T = \begin{pmatrix} 1 & 1 & 0 \\ 0 & 1 & 1 \\ 0 & 0 & 1 \end{pmatrix}.$

        \item[1.4.] $\G_{4,3,4(\lambda,\varphi)}:$\,
$ad_{{X}_1} = \begin{pmatrix}\cos\varphi & -\sin\varphi & 0 \\ \sin\varphi & \cos\varphi & 0 \\ 0 & 0 & \lambda \end{pmatrix}; \lambda \in \R \setminus \lbrace 0\rbrace, \varphi \in (0, \pi).$
	  \end{itemize}
	  
	\item[2.] $\G^1 = Span(X,Y,Z) = \mathfrak{h}_3, \, ad_T \in 
End_{\R} \left(\G^1\right) \equiv Mat_3(\R)$
	  \begin{itemize}
        \item[2.1.] $\G_{4,4,1}:$ \,
$ad_T = \begin{pmatrix}
0& 1& 0 \\ -1 & 0 & 0 \\ 0 & 0 & 0 \end{pmatrix}.$
        \item[2.2.] $\G_{4,4,2} = {\R}.\mathfrak{h}_3$ (the 4-dimensional Diamond Lie algebra) $:$ 
$$ad_T = \begin{pmatrix}
-1 & 0 & 0 \\ 0 & 1 & 0 \\ 0 & 0 & 0 \end{pmatrix}.$$
	  \end{itemize}
\end{itemize}
\end{prop}


\subsection{Classification of $MD(5,1)$-class and $MD(5,4)$-class}

\begin{prop}[{\bf Classification of $MD(5,1)$-algebras, see \cite{LS08, LHT11}}]\label{prop3.3}
Let $\G$ be an $MD(5,1)$-algebra, Then we can choose a suitable basis $(X_{1}, X_{2}, X_{3},$ $X_{4}, X_{5})$ of $\G$ such that $\G^1 = Span(X_5) = \R.X_5$ and $\G$ is isomorphic to one of the following Lie algebras.
    \begin{itemize}
    	\item[1.]  $\G_{5,1,1} = \mathfrak{h}_5$ (the 5-dimensional real Heisenberg Lie algebra)$: [X_{1}, X_{2}]=[X_{3},X_{4}]=X_{5};$ the other Lie brackets are trivial. In this case, $\G$ is indecomposable.
        \item[2.] $\G_{5,1,2} = {\rm aff} (\R) \oplus \R.X_1 \oplus \R.X_2 \oplus \R.X_3: [X_4, X_5]= X_5;$ the other Lie brackets are trivial. In this case, $\G$ is decomposable.
    \end{itemize}
\end{prop}


\begin{prop}[{\bf Classification of $MD(5,4)$-algebras, see \cite{LS08}}]\label{prop3.4}
	Let $\G$ be an $MD(5,4)$-algebra. Then $\G$ must be indecomposable, and $\G^1$ is commutative. Moreover, we can choose a suitable basis $(X_1, X_2, X_3,X_4,X_5)$ of $\G$ such that $\G^1 = Span(X_2, X_3,X_4, X_5) \equiv \R^4$, $ad_{X_1} \in Aut(\G^{1}) \equiv GL_{4}(\R)$ and $\G$ is isomorphic to one of the following Lie algebras. 
   \begin{itemize}
     \item[4.1.]${\mathcal{G}}_{5,4,1({\lambda}_{1}, {\lambda}_{2}, {\lambda}_{3})}:
 ad_{{X}_1} = \begin{pmatrix} {{\lambda}_1}&0&0&0\\
 0&{{\lambda}_2}&0&0\\0&0&{\lambda}_{3}&0\\0&0&0&1\end{pmatrix};
  {\lambda}_1, {\lambda}_2, {\lambda}_3 \in \mathbb{R}\setminus
  \lbrace 0, 1\rbrace;$\,\,\, ${\lambda}_1 \neq {\lambda}_2 \neq {\lambda}_3 \neq
  {\lambda}_1.$
     \item[4.2.]${\mathcal{G}}_{5,4,2({\lambda}_{1}, {\lambda}_{2})}:\,
ad_{{X}_1} = \begin{pmatrix} {{\lambda}_1}&0&0&0\\
0&{{\lambda}_2}&0&0\\0&0&1&0\\0&0&0&1\end{pmatrix};\,
{\lambda}_{1}, {\lambda}_{2} \in \mathbb{R}\setminus \lbrace 0, 1
\rbrace , {\lambda}_1 \neq {\lambda}_2. $

        \item[4.3.]${\mathcal{G}}_{5,4,3(\lambda)}:\,
 ad_{{X}_1} = \begin{pmatrix}
 {\lambda}&0&0&0\\0&{\lambda}&0&0\\0&0&1&0\\0&0&0&1 \end{pmatrix}; \,
 {\lambda} \in \mathbb{R}\setminus \lbrace 0, 1 \rbrace .$

        \item[4.4.]${\mathcal{G}}_{5,4,4(\lambda)}:\,
ad_{{X}_1} = \begin{pmatrix} {\lambda}&0&0&0\\0&1&0&0\\
0&0&1&0\\0&0&0&1 \end{pmatrix};\quad {\lambda} \in
\mathbb{R}\setminus \lbrace 0, 1 \rbrace.$

        \item[4.5.]${\mathcal{G}}_{5,4,5}:\,
ad_{{X}_1} = \begin{pmatrix} 1&0&0&0\\0&1&0&0\\
0&0&1&0\\0&0&0&1 \end{pmatrix}.$

        \item[4.6.]${\mathcal{G}}_{5,4,6({\lambda}_{1}, {\lambda}_{2})}:\,
ad_{{X}_1} = \begin{pmatrix} {{\lambda}_1}&0&0&0\\
0&{{\lambda}_2}&0&0\\0&0&1&1\\0&0&0&1\end{pmatrix};
{\lambda}_{1}, {\lambda}_{2} \in \mathbb{R}\setminus \lbrace 0, 1
\rbrace , {\lambda}_1 \neq {\lambda}_2.$

        \item[4.7.]${\mathcal{G}}_{5,4,7(\lambda)}:\,
ad_{{X}_1} = \begin{pmatrix}
{\lambda}&0&0&0\\0&{\lambda}&0&0\\0&0&1&1\\0&0&0&1 \end{pmatrix};
\quad {\lambda} \in \mathbb{R}\setminus \lbrace 0, 1 \rbrace .$

        \item[4.8.]${\mathcal{G}}_{5,4,8(\lambda)}:\,
ad_{{X}_1} = \begin{pmatrix}
{\lambda}&1&0&0\\0&{\lambda}&0&0\\0&0&1&1\\0&0&0&1 \end{pmatrix};
\quad {\lambda} \in \mathbb{R}\setminus \lbrace 0, 1 \rbrace .$

        \item[4.9.]${\mathcal{G}}_{5,4,9(\lambda)}:\,
ad_{{X}_1} = \begin{pmatrix}
{\lambda}&0&0&0\\0&1&1&0\\0&0&1&1\\0&0&0&1 \end{pmatrix}; \quad
{\lambda} \in \mathbb{R}\setminus \lbrace 0, 1\rbrace .$

        \item[4.10.]${\mathcal{G}}_{5,4,10}:\,
ad_{{X}_1} = \begin{pmatrix} 1&1&0&0\\0&1&1&0\\
0&0&1&1\\0&0&0&1 \end{pmatrix}.$

        \item[4.11.]${\mathcal{G}}_{5,4,11({\lambda}_{1}, {\lambda}_{2},\varphi)}:$
$$ad_{{X}_1} = \begin{pmatrix} \cos\varphi&-\sin\varphi&0&0\\
\sin\varphi& \cos\varphi&0&0\\0&0&{\lambda}_{1}&0\\0&0&0&{\lambda}_{2}\end{pmatrix};\, {\lambda}_{1}, {\lambda}_{2} \in \mathbb{R}\setminus \lbrace 0 \rbrace ,
{\lambda}_1 \neq {\lambda}_2,\varphi \in (0,\pi).$$

        \item[4.12.]${\mathcal{G}}_{5,4,12(\lambda, \varphi)}:$
$$ad_{{X}_1} = \begin{pmatrix} \cos\varphi&-\sin\varphi&0&0\\
\sin\varphi&\cos\varphi&0&0\\0&0&\lambda&0\\0&0&0&\lambda\end{pmatrix};\quad
\lambda \in \mathbb{R}\setminus \lbrace 0 \rbrace,\varphi \in (0,\pi).$$

        \item[4.13.]${\mathcal{G}}_{5,4,13(\lambda, \varphi)}:$
$$ad_{{X}_1} = \begin{pmatrix} \cos\varphi&-\sin\varphi&0&0\\
\sin\varphi&\cos\varphi&0&0\\0&0&\lambda&1\\0&0&0&\lambda\end{pmatrix};\quad
\lambda \in \mathbb{R}\setminus \lbrace 0 \rbrace,\varphi \in (0,\pi).$$

        \item[4.14.]${\mathcal{G}}_{5,4,14(\lambda, \mu, \varphi)}:$
$$ad_{{X}_1} = \begin{pmatrix} \cos\varphi&-\sin\varphi&0&0\\
\sin\varphi&\cos\varphi&0&0\\0&0&\lambda&-\mu\\0&0&\mu&\lambda\end{pmatrix};\quad
\lambda, \mu \in \mathbb{R}, \mu > 0, \varphi \in (0,\pi).$$
   \end{itemize}
\end{prop}

In the next subsection, we introduce one noticeable result of D. Arnal, M. Cahen and J. Ludwig \cite{ACL95} in 1995.


\subsection{List of $MD$-algebras whose simply connected $MD$-groups have only coadjoint orbits of dimension zero or two}

In an attempt to classify solvable Lie algebras by structure, in 1995, D. Arnal, M. Cahen and J. Ludwig \cite{ACL95} have listed, up to a direct central factor, all Lie algebras (solvable or not) such that the maximal dimension of $K$-orbits of corresponding connected and simply connected Lie groups is just two. However, they have not yet classified, up to isomorphism, these algebras.   

\begin{prop}[{\bf see \cite{ACL95}}]\label{prop3.5}
	Let $G$ be a connected, simply connected solvable Lie group whose coadjoint orbits have dimension smaller or equal to two. Let $\G$ be the Lie algebra of $G$. Then, up to a direct central factor, $\G$ belongs to the following list of algebras:
\begin{itemize}
	\item[(i)]$\R.T \oplus\mathfrak{a}$ where $\mathfrak{a}$ is an abelian ideal and $ad_T \in End(\mathfrak{a}).$
	\item[(ii)] $\R.T \oplus \mathfrak{h}_3$ where $\mathfrak{h}_3$ is the 3-dimensional Heisenberg algebra spanned by $(X, Y, Z)$ with $[X, Y] = Z$ and
		\begin{itemize}
			\item[$\bullet$] either $[T, X] = X, \, [T, Y] = - Y, \, [T, Z] = 0$ (the 4-dimensional Diamond algebra).
			\item[$\bullet$] or $[T, X] = Y, \, [T, Y] = - X, \, [T, Z] = 0$.
		\end{itemize}
	\item[(iii)] $\G$ is 5-dimensional with basis $\left( X_1, X_2, X_3, Y_1, Y_2\right)$ and the multiplicative law reads
		\[ [X_1, X_2] = X_3, \, [X_1, X_3] = Y_1, \, [X_2, X_3] = Y_2.\]
	\item[(iv)]$\G$ is 6-dimensional with basis $\left( X_1, X_2, X_3, Y_1, Y_2, Y_3\right)$ and the nonvanishing brackets are
		\[ [X_1, X_2] = Y_3, \, [X_2, X_3] = Y_1, \, [X_3, X_1] = Y_2.\]
\end{itemize}
\end{prop}

\begin{rem}
Clearly we have two following remarks.
\begin{itemize}
	\item[(i)] There is an infinite family of non-isomorphic $MD$-algebras in part (i) of Proposition \ref{prop3.5}. Namely, the part (i) of Proposition \ref{prop3.5} includes all $MD(4,1)$-algebras, $MD(5,1)$-algebras, $MD(5,4)$-algebras except the 5-dimensional Heisenberg Lie algebra. Furthermore, the last two $MD(4,3)$-algebras in Proposition \ref{prop3.2} coincide with two algebras in the part (ii) of Proposition \ref{prop3.5}, but remaining four $MD(4,3)$-algebras in Proposition \ref{prop3.2} are included in the part (i) of Proposition \ref{prop3.1}. 
	\item[(ii)] However, it should be noted that the indecomposable $MD(2,1)$-algebras\, $\text{aff}(\R)$ and \,$\mathfrak{h}_3$; the decomposable $MD(4,1)$-algebras $\mathfrak{h}_3 \oplus \R.T$ (in Proposition \ref{prop3.1}), $\text{aff}(\R) \oplus \R.Z \oplus \R.T$ and the decomposable $MD(5,1)$-algebra $\text{aff}(\R) \oplus \R.X_1 \oplus \R.X_2 \oplus \R.X_3$ are not included in the list of Proposition \ref{prop3.5}, although it is obvious that all coadjoint orbits of the Lie groups corresponding to $\mathfrak{h}_3 \oplus \R.T, \,\text{aff}(\R) \oplus \R.Z \oplus \R.T, \, \text{aff}(\R) \oplus \R.X_1 \oplus \R.X_2 \oplus \R.X_3$ have dimension zero or two. So that was one shortcoming in Arnal's list.
\end{itemize} 
\end{rem}


\section{Classification of $MD(*, 1)$-class and $MD(*,* - 1)$-class}

Now we will introduce the complete classification, up to an isomorphism, of all $MD$-algebras (of arbitrary dimension) having the first derived ideal of dimension one or codimension one. These results are generalizations of Propositions \ref{prop3.1}, \ref{prop3.2}, \ref{prop3.3} and \ref{prop3.4} in Section \ref{Sect3}. 
 
\subsection{The main results}

	\begin{thm}[{\bf The Complete Classification of $MD(*,1)$-class}]\label{thm4.1}
		$MD(*,1)$-class coincides with the class of all real solvable Lie algebras whose the first derived ideal is 1-dimensional, moreover $MD(*,1)$ includes only the Lie algebra of the group of affine transformations of the real straight line, the real Heisenberg Lie algebras and their direct extensions by the real commutative Lie algebras. In other words, if $\G$ is a $n$-dimensional real solvable Lie algebra whose the first ideal $\G^1:= [\G, \G]$ is 1-dimensional $(2 \leqslant n \in \mathbb{N})$ then $\G$ is an $MD(n,1)$-algebra and $\G$ is isomorphic to one and only one of the following Lie algebras.
		\begin{itemize}
		\item[(i)] The Lie algebra ${\rm aff} (\R)$ of the group ${\rm Aff} (\R)$ of all affine transformations on $\R$; $n = 2$.
		\item[(ii)] ${\rm aff} (\R) \oplus \R^{n-2}$; $3 \leqslant n.$
		\item[(iii)] The real Heisenberg Lie algebra $\mathfrak{h}_{2m + 1}$;\, $3 \leqslant n = 2m + 1.$
		\item[(iv)] $\mathfrak{h}_{2m+1} \oplus \R^{n-2m-1}$; $3 \leqslant 2m + 1 < n.$
		\end{itemize}
	\end{thm}

It is clear that Theorem \ref{thm4.1} can be formulated by another way in the following consequence which gives a new character of the real Heisenberg Lie algebras.

\begin{cor}[{\bf A New Character of the Real Heisenberg Lie Algebras}]
Let $\G$ be a real Lie algebra of dimension $n \, (3 \leqslant n \in \mathbb{N})$. Then the following conditions are equivalent.
\begin{itemize}
	\item[(i)] $\G$ is indecomposable and has the first derived ideal $\G^1 = [\G, \G] \cong \R$.
	\item[(ii)] $\G$ is an indecomposable $MD(n,1)$-algebra.
	\item[(iii)] $\G$ is the n-dimensional Heisenberg Lie algebra (in particular, $n$ is odd).
\end{itemize} 
\end{cor}

The next theorem gives one necessary and sufficient condition to recognize one $MD(n, n-1)$-algebra $(4 \leqslant n \in \mathbb{N})$.

\begin{thm}[{\bf Necessary and Sufficient Conditions to identify $MD(*, *-1)$-alge\\bras}]\label{thm4.3}
	Let $\G$  be a real solvable Lie algebra of dimension $n \, (3 \leqslant n \in \mathbb{N})$ such that its first derived ideal $\G^1$ is $(n-1)$-dimensional.
	\begin{itemize}
		\item[(i)] If $\G^1$ is commutative then $\G$ is an $MD(n, n-1)$-algebra, moreover $\G$ is indecomposable.
		\item[(ii)] If $n > 4$ and $\G$ is an $MD(n, n-1)$-algebra then $\G^1$ is commutative.
	\end{itemize}
\end{thm}

\begin{rem}
	When $n \leqslant 4$, assertion (ii) is not true. Namely, if $n < 4$, all the $n$-dimensional Lie algebras are $MD$-algebras, and moreover, they can be listed easily. If $n = 4$, as previously indicated, the derived ideal of the 4-dimensional real Diamond Lie algebra is the 3-dimensional Heisenberg algebra which is non commutative and 1-codimensional. In fact, all $MD4$-algebras were completely classified in 1990 by A. V. Le \cites{Le90-2,Le93}, and the classification of $MD(4,1)$-class and $MD(4,3)$-class were recalled in Propositions \ref{prop3.1} and \ref{prop3.2}.  
\end{rem}

The last theorem will characterize every $MD(n, n-1)$-algebra by an invertible real $(n-1)$-square matrix and reduces the task of classifying $MD(n, n-1)$-class to obtaining the well-known classification of equivalent of proportional similar matrices. Let $\G$ be an $MD(n, n-1)$-algebra $(3 \leqslant n \in \mathbb{N})$ generated by a basis $(X_1,X_2,\ldots,X_n)$ such that the first derived ideal $\G^1$ is 1-codimensional and spanned by $(X_1, X_2, \ldots, X_{n-1})$. It is obviously that the Lie structure of $\G$ is well understood by the invertible real $(n-1)$-square matrix of map $ad_{X_n}$ considering as an automorphism of $\G^1$ for the basis $(X_1, X_2, \ldots, X_{n-1})$. 
 
\begin{thm}\label{thm4.5}
		Let $\G$ be a real vector space of dimension $n \, (3 \leqslant n \in \mathbb{N})$ generated by a basis $(X_1,X_2,\ldots, X_n)$ and $\G^1: = Span(X_1, X_2, \ldots, X_{n-1})$ is the 1-codimensional subspace of $\G$. Then we have the following assertions.
	\begin{itemize}
		\item[(i)] Each invertible real $(n-1)$-square matrix $A$ always defines one Lie structure on $\G$ such that $\G$ is an $MD(n, n-1)$-algebra with the first derived ideal is commutative, exactly equals to $\G^1$ and $A$ is exactly the matrix of adjoint map $ad_{X_n}$ on $\G^1$ in the chosen basis $(X_1, X_2, \ldots,X_{n-1})$.
		\item[(ii)] Two invertible real $(n-1)$-square matrices $A, B$ define two Lie structures on $\G$ which are isomorphic if and only if there exist a non-zero real number $c$ and an invertible real $(n-1)$-square matrix $C$ so that $cA = CBC^{-1}$.
	\end{itemize}
\end{thm}

\begin{rem} In view of Theorem \ref{thm4.5}, we have the following remarks.
  \begin{itemize}     
    	\item[(i)] Two invertible real square matrices $A, B$ of the same order are called \emph{proportional similar} if (and only if) there exist a non-zero real number $c$ and an invertible real square matrix $C$ of the same order as  of $A, B$ so that $cA = CBC^{-1}$. In fact, assertion (ii) of Theorem \ref{thm4.5} gives the classification of $MD(n,n-1)$-algebras ($n > 4$) by using the well-known classification of invertible real matrices in proportional similar equivalent relation.
    	\item[(ii)] The classification of indecomposable $MD(5, 4)$-algebras in Proposition \ref{prop3.4} of this paper gives one concrete illustration of Theorem \ref{thm4.5} when $n = 5$. On principle, it is not hard to list all non-isomorphic indecomposable $MD(n,n-1)$-algebras by applying Theorem \ref{thm4.5} for $n$ is small, for example $n = 6, 7, ... $ \,.
    \end{itemize}
\end{rem}

\subsection{Proof of Theorem \ref{thm4.1}}

	In this section, we always consider $\G$ as a real solvable Lie algebra of dimension $n \geqslant 3$ whose the derived ideal $\G^1 = [\G,\G]$ is 1-dimensional, in particular $\G^1$ is commutative. Without loss of generality, we can choose a suitable basis such that 
	$$\G = Span \left( X_1, X_2, \ldots, X_n \right),\G^1 = Span \left(X_n \right) = \R.X_n.$$
Let $G$ be the connected, simply connected Lie group corresponding to $\G$.

	When $n = 2$, it is obvious that the part (i) in Theorem \ref{thm4.1} holds because ${\rm aff} (\R)$ is an $MD(2,1)$-algebra (see Remark \ref{rem2.11}) and it is the unique non-commutative real Lie algebra of dimension 2. Therefore, in the rest of this subsection, we can suppose that $n \geqslant 3$. Denote
$$[X_i,X_n] = a_iX_n;\, [X_i,X_j] = a_{ij}X_n\, (a_i, a_{ij} \in \R);\, i,j = 1, 2, \ldots, n-1.$$
Evidently, the Lie structure on $\G$ is well understood by the vector 
$$\mathbf{a}: = (a_1, a_2, \ldots, a_{n-1})$$
and the skew-symmetric real $(n-1)$-square matrix $A: = (a_{ij})_{i,j=\overline{1,n-1}}$. There are two cases to consider for the values of the vector $\mathbf{a}$:\,$\mathbf{a} = 0$ or $\mathbf{a} \neq 0$.

\noindent{\bf The first case: a $\neq$ 0.} \, First, we consider the case $\mathbf{a} \neq 0$, i.e. $\exists i \in \{1, \ldots , n-1\}$ such that $a_i \neq 0$, that means $X_{n}$ is not in the center $Z(\G)$ of $\G$. Renumber the chosen basis, if necessary, we can always suppose that $a_{n-1} \neq 0$. Then $[X_{n-1}, X_n] = a_{n-1} X_n \neq 0$. In this case, we will show that $\G$ is a trivial extension of the Lie algebra ${\rm aff} (\R)$. Namely, we have the following lemma.

\begin{lem}\label{lem4.7}
	If $\mathbf{a} \neq 0$ then $\G$ is an $MD(n,1)$-algebra which is isomorphic to ${\rm aff} (\R) \oplus \R^{n-2}$ when $n \geqslant 3$.	
\end{lem}

\begin{proof}
	Using the following change of basis 
		\[
			Y_i = X_i-\frac{a_i}{a_{n-1}} X_{n-1},\, i = 1, 2, \ldots , n-2; \, Y_{n-1} = \frac{1}{a_{n-1}} X_{n-1}, \, Y_n = X_n
		\]
we get 
		\[
			[Y_i, Y_n] = 0; i = 1, 2, \ldots , n-2; \, [Y_{n-1}, Y_n] = Y_n.
		\]
Hence, without loss of generality, we can now assume 
		\[
			[X_i, X_n] = 0; \, i = 1, 2, \ldots , n-2; \, [X_{n-1}, X_n] = X_n.
		\]
Using the Jacobi identities for triples $(X_i, X_j, X_{n-1}), \,1 \leqslant i < j \leqslant n-2$, we get
		\[
			\begin{array}{l l}
				& [[X_i, X_j], X_{n-1}] + [[X_{n-1}, X_i], X_j] + [[X_j, X_{n-1}], X_i] = 0\\
				\Rightarrow & a_{ij}[X_n, X_{n-1}]  = 0 \Rightarrow -a_{ij}X_n = 0 \Rightarrow a_{ij} = 0 \\
				\Rightarrow & [X_i, X_j] = a_{ij}X_n = 0; \,\, i, j = 1, 2, \ldots , n - 2.
			\end{array}
		\]
Now, using the change of basis as follows 
		\[
			Z_i = X_i + a_{i,n-1}X_n; \, i = 1, 2, \ldots, n - 2; \, Z_{n-1} = X_{n-1}, \, Z_n = X_n
		\]
we get $[Z_i, Z_{n-1}] = 0; \, i = 1, 2, \ldots, n-2$. So we can suppose now that 
		\[
			[X_i, X_{n-1}] = 0, \, i = 1, 2, \ldots , n-2.
		\]
Hence, in this case, $\G$ is isomorphic to the following Lie algebra
		\[
			\begin{array}{c}
				{\rm aff} (\R) \oplus \R^{n-2}= Span \left(X_1, X_2, \ldots , X_n\right),\, [X_{n-1}, X_n] = X_n, 		
			\end{array}
		\]
where the other Lie brackets are trivial. Obviously, the $K$-orbits of $G$ is of dimension 0 or 2. This means that $\G$ is an $MD(n,1)$-algebra.
\end{proof}

\noindent{\bf The second case: a = 0.} \, Now, we consider the second case $\mathbf{a} = 0$, i.e. $[X_i, X_n] = 0$ for all $i = 1, 2, \ldots, n-1$, in particular $X_n \in Z(\G)$. Then the Lie structure of $\G$ is uniquely defined by the skew-symmetric real $(n-1)$-square matrix $A = (a_{ij})_{i,j=\overline{1,n-1}}$, which is called the \emph{structure matrix} of $\G$. Since $\G^1 = Span(X_n)$ is 1-dimensional, $A$ is non-trivial and $0 < {\rm rank} A$ is even. We have the following lemma.

\begin{lem}\label{lem4.8}
	If $\mathbf{a} = 0$, i.e. $[X_i, X_n] = 0$ for all $i = 1, 2, \ldots , {n-1}$, then the Lie algebra $\G$ is an $MD(n,1)$-algebra and the maximal dimension of the $K$-orbits of $G$ is the rank of the structure matrix $A$.
\end{lem}

\begin{proof}
	Let $\G^*\equiv \R^n$  be the dual space of $\G$ with dual basis $\left(X_1^*, X_2^*, \ldots , X_n^*\right)$ and $F = f_1X_1^* + f_2X_2^* + \ldots  + f_nX_n^* \equiv (f_1, f_2, \ldots , f_n)$ be an arbitrary element of $\G^*$. The Kirillov form $B_F$ is given as follows
		\[
			B_F: = \left(\langle F, [X_i, X_j] \rangle\right)_{i,j = \overline{1,n}}
		 = f_n \begin{bmatrix}
		 			{{a_{11}}}& \ldots &{{a_{1,{n-1}}}}&0\\	
					\vdots & \ldots & \vdots & \vdots \\
					{{a_{{n-1},1}}}& \ldots &{{a_{{n-1},{n-1}}}}&0\\
					0& \ldots &0&0
				\end{bmatrix}
		= f_n \begin{bmatrix}
					A & 0 \\ 0 & 0
				\end{bmatrix}
		\]
and	${\rm rank} B_F \in \{0, 2k\}$ where $2k = {\rm rank} A$ is the rank of the structural matrix. More precisely
\begin{itemize}
	\item ${\rm rank} B_F = 0$ if and only if $f_n = 0$, i.e. $F = (f_1, f_2, \ldots, f_{n-1}, 0)$.
	\item ${\rm rank} B_F = {\rm rank} A = 2k > 0$ if and only if $f_n \neq 0$.
\end{itemize}
Hence, in view of Proposition \ref{prop2.4}, $\G$ is an $MD(n,1)$-algebra and the maximal dimension of $K$-orbits of $G$ is the rank of the structure matrix.
\end{proof}

Now, we will consider whether $\G$ is decomposable in the second case.

\begin{lem}\label{lem4.9}
	If $a = 0$, i.e. $[X_i, X_n] = 0$ for all $i = 1, 2, \ldots, n-1$, then the Lie algebra $\G$ is decomposable if and only if the dimension of the center of $\G$ is greater than 1.
\end{lem}

\begin{proof}
	Denote $Z(\G)$ to be the center of $\G$. Obviously, $X_n$ is in $Z(\G)$ because of $[X_i, X_n] = 0$ for all $i = 1, 2, \ldots, {n-1}$, i.e. $\dim Z(\G) > 0$.
	
\noindent $(\Longrightarrow)$ Suppose that $\G$ is decomposable, i.e. $\G = \A \oplus \B$ in which $\A, \B$ are non-trivial proper Lie subalgebras of $\G$. Put $X_n = X_a + X_b \in Z(\G)$ with some $X_a \in \A$ and some $X_b \in \B$. Let us consider an arbitrary element $Y = Y_a + Y_b \in \G$ with $Y_a \in \A,\, Y_b \in \B$. We have
	\[
		\begin{array}{l}
			0 = [X_n, Y] = [X_a + X_b, Y_a + Y_b] = [X_a, Y_a] + [X_b, Y_b]\\
			\Rightarrow [X_a, Y_a] = [X_b, Y_b] = 0 \Rightarrow X_a, X_b \in Z(\G).
		\end{array}
	\]
\begin{itemize}
	\item If $X_a \neq 0 \neq X_b$ then they are of course linear independent and $\dim Z(\G) > 1$.
	\item If $X_a$ or $X_b$ is 0. Without loss of generality, we can suppose that $X_a = 0$, i.e. $X_n = X_b \in \B$. In particular $\G^1 = Span(X_n) \subseteq \B$. Let $X \neq 0 \in \A$ be an arbitrary element. Obviously, $[X, Z] = 0$ for every $Z \in \B$. On the other hand, we have
	\[
		[X, T] \in \A \cap \G^1 \subseteq \A \cap \B = 0 \Rightarrow [X, T] = 0, \forall T \in \A.
	\]
This means that $X$ commutes with any element of $\G = \A \oplus \B$, i.e. $X \in Z(\G)$. Because of $X \in \A, X_n \in \B$ so $X, X_n$ are linear independent and $\dim Z(\G) > 1$.
\end{itemize}
Hence, $\dim Z(\G) > 1$ in any case.

\noindent $(\Longleftarrow)$ Suppose $\dim Z(\G) > 1$. There exists $X \in Z(\G)$ such that $X, X_n$ are independent. We can add $T_1, \ldots, T_{n-2}$ in $(X, X_n)$ to get a new basis of $\G$. Then we have 
	\[
		\G = Span(X) \oplus Span(X_n, T_1, \ldots , T_{n-2}).
	\]
Therefore $\G$ is decomposable.
\end{proof}

\begin{rem}
	The center of the Heisenberg Lie algebra is 1-dimensional, so its indecomposableness is unsurprised.
\end{rem}

Recall that each $MD(n,1)$-algebra $\G$ in the second case is always defined uniquely by an $(n-1)$-square (skew-symmetric) structure matrix $A$. Now we will consider whether two structure $A$ and $B$ give us isomorphic Lie algebras.

\begin{lem}\label{lem4.11}
	Let $A = (a_{ij})_{i,j=\overline{1,n-1}},  B = (b_{ij})_{i,j=\overline{1,n-1}}$ be skew-symmetric real $(n-1)$-square matrices and $\G_A, \G_B$ be $MD(n,1)$-algebras which are defined by $A, B$ respectively. Then
	\[
		(\G_A \cong \G_B) \Leftrightarrow (\exists c\in \R^*, \exists C \in GL_{n-1}(\R) \,\, \text{such that} \,\, cA = C^TBC),
	\]
where $C^T$ is the transpose of $C$.
\end{lem}

\begin{proof}
	$(\Longrightarrow)$ Let $f: \G_A \to \G_B$ be an isomorphism. Since $f \left(\G_A^1\right)=\G_B^1$, there is a non-zero real number $c$ so that $f(X_n) = cX_n$. Clearly the matrix of $f$ in the basis $(X_1, X_2, \ldots , X_{n-1}, X_n)$ is given as follows
		\[M = \left[ {\begin{array}{*{20}{c}}
		{{c_{11}}}& \cdots &{{c_{1,{n-1}}}}&0\\
 		\vdots & \cdots & \vdots & \vdots \\
		{{c_{{n-1},1}}}& \cdots &{{c_{{n-1},{n-1}}}}&0\\
		{{c_{n1}}}& \cdots &{{c_{n,{n-1}}}}&c
		\end{array}} \right] = \left[ {\begin{array}{*{20}{c}}
		C&0\\
		*&c
		\end{array}} \right]\]
in which $C = (c_{ij})_{i,j=\overline{1,n-1}}$ is a real $(n-1)$-square matrix, and $*$ is the vector $(c_{n1}, \ldots , c_{n, {n-1}})$. Because $f$ is an isomorphism, $M$ is invertible and so is $C$. Hence, the linear map $f$ is a Lie isomorphism if and only if
$$\begin{array}{*{20}{l}} & f([X_i, X_j]_A ) = [f(X_i), f(X_j)]_B;  \quad \forall\, i, j = 1, 2, \ldots , {n-1} \\
	    \Leftrightarrow  & f(a_{ij}X_n) = \left[\sum\limits_{k=1}^{n-1}c_{ki}X_k + c_{ni}X_n,\sum\limits_{l = 1}^{{n-1}}
	                   {{c_{lj}}{X_l}}   + c_{nj}X_n \right]_B; \quad \forall\,  i, j = 1, 2, \ldots , {n-1} \\
        \Leftrightarrow & ca_{ij}X_n  =\sum\limits_{k = 1}^{{n-1}} {{c_{ki}}} .\sum\limits_{l = 1}^{{n-1}} {{c_{lj}}
                       }[X_k, X_l]_B; \quad \forall\,  i, j = = 1, 2, \ldots , {n-1} \\
        \Leftrightarrow & ca_{ij}X_n = \left(\sum\limits_{k,l = 1}^{{n-1}} {c_{ki}} .{b_{kl}}.{c_{lj}}
                       \right)X_n; \quad \forall\,  i,j = 1, 2, \ldots , {n-1} \\    
        \Leftrightarrow  & ca_{ij} = \sum\limits_{k,l = 1}^{{n-1}} {c_{ki}}.{b_{kl}}.{c_{lj}}; 
                        \quad \forall \, i, j = 1, 2, \ldots, {n-1} \\
        \Leftrightarrow  & cA = C^TAC. \end{array}$$    
                                                                  
\noindent$(\Longleftarrow)$ Conversely, suppose that there exist a non-zero real number $c$ and an invertible real $(n-1)$-square matrix $C =(c_{ij})_{i,j=\overline{1,n-1}}$ such that $cA = C^TBC$. Let $f: \G_A \to \G_B$ be a linear map which is defined, in the basis $(X_1, X_2, \ldots, X_n)$, by the matrix $M'$ as follows
	\[
		M' = \begin{bmatrix}
					c_{11} & \cdots & c_{1,n-1} & 0\\
					\vdots & \cdots & \vdots & \vdots \\
					c_{n-1,1} & \cdots & c_{n-1,n-1} & 0\\
					0 & \cdots & 0 & c
				\end{bmatrix}
			= \begin{bmatrix} C & 0 \\ 0 & c \end{bmatrix}.
	\]
Since $C$ is invertible and $c \neq 0$, $f$ is a linear isomorphism. Moreover, it is easy to check that $f$ is also a Lie homomorphism. Therefore, $f$ is a Lie isomorphism.
\end{proof}

\begin{rem}
	Recall that two real $(n-1)$-square matrices $A, B$ are said to be \emph{congruent} if there exists an invertible $(n-1)$-square matrix $C$ such that $B = C^TAC$. Furthermore, any non-zero skew-symmetric real square matrix can be always transformed into the canonical form. More precisely, for any non-zero skew-symmetric real $(n-1)$-square matrix $A$, there exists a real orthogonal matrix $C$ such that
	$$C^TAC = {\rm diag} (\Lambda_1, \Lambda_2, \ldots , \Lambda_m, 0, \ldots, 0)$$
where $\Lambda_j: =  \begin{bmatrix} 0 & \lambda _i \\  -\lambda _i & 0 \end{bmatrix}$ and $\lbrace \pm i\lambda_1, \ldots , \pm i\lambda_m \rbrace$ ($i$ is the imaginary unit in the complex field $\mathbb{C}$) is the set of all multiple eigenvalues of $A$. For example, the real Heisenberg Lie algebra
	$$\mathfrak{h}_{2m + 1}: = \langle X_i, Y_i, Z: i = 1, 2, \ldots , m \rangle; [X_i, Y_i] = Z, i = 1, 2, \ldots , m$$
has the structure matrix $H = {\rm diag} (I, \ldots, I)$ including $n$ blocks $I = \begin{bmatrix} 0 & 1\\ -1 & 0 \end{bmatrix}$. This matrix $H$ has exactly two $m$-multiple eingenvalues $\pm i$ and $H$ has no eigenvalue 0.
\end{rem}


\begin{proof}[{\bf Proof of Theorem \ref{thm4.1}}]
	Recall that we need only show the part (ii), (iii), (iv) of Theorem \ref{thm4.1}. Lemmas \ref{lem4.7}, \ref{lem4.8} and \ref{lem4.9} show that the considered Lie algebra $\G$ belongs to $MD(n,1)$-class. Moreover, the part (ii) of Theorem \ref{thm4.1} is implied directly from Lemma \ref{lem4.7}. We only need to prove the part (iii) and (iv).
 
	In the basis $(X_1, X_2, \ldots , X_n)$, the structure matrix of $\G$ is $A$. We will choose a new basis to get the standard form $B = C^TAC$. By Lemma \ref{lem4.11}, the Lie algebra defined by the matrix $B$ is isomorphic to $\G$.

	If $B$ has no zero eigenvalue, i.e. $B = {\rm diag} (\Lambda_1, \Lambda_2, \ldots , \Lambda_m),\, 2m = {n-1}$. Put
	\[
		D: = {\rm diag} \left( 1, \frac{1}{{{\lambda _1}}}, 1, \frac{1}{\lambda_2} , \ldots , 1, \frac{1}{\lambda_m} \right).
	\]
Then we get the structure matrix $H = D^TBD$ of the $(2m + 1)$-dimensional real Heisenberg algebra $\mathfrak{h}_{2m + 1},\, 2m + 1 = n$. By Lemma \ref{lem4.11}, $\G$ is isomorphic to $\mathfrak{h}_{2m + 1}$. So the part (iii) is proved.
	
	If $B$ has eigenvalues 0, i.e. $0 < 2m < {n-1}$, then $Z(\G)$ is generated by the basis $X_{2m + 1}, \ldots , X_n$ whose dimension is greater than 1. By Lemma \ref{lem4.9}, $\G$ is decomposable, namely $\G$ is isomorphic to $\mathfrak{h}_{2m + 1} \oplus \R^k$ where $k = n-(2m + 1) > 0$. Actually, the direct summand $\R^k$ is the commutative Lie subalgebra of $\G$ generated by $(X_{2m + 2}, \ldots , X_n)$ and $\mathfrak{h}_{2m + 1}$ is generated by $(X_1, X_2, \ldots, X_{2m}; X_{2m + 1})$. So the part (iv) is proved and the proof of Theorem \ref{thm4.1} is complete.
\end{proof}
 

\subsection{Proof of Theorem \ref{thm4.3}}

In this section, we always consider $\G$ as a real solvable Lie algebra of dimension $n \geqslant 3$ with the 1-codimensional first derived ideal $\G^1 = [\G,\G]$. Assume that $\dim \G^2 = \dim([\G^1, \G^1]) = k \leqslant n-2$. Without loss of generality, we can choose a suitable basis such that 
$$\canbangg{\G\,\,\,= Span \left(X_1, X_2, \ldots , X_n \right), \\
\G^1 = [\G,\G] = Span \left(X_1, X_2, \ldots , X_{n-1} \right), \\
\G^2 = \left[\G^1,\G^1\right] = Span \left(X_1, \ldots , X_k \right), \, k \leqslant  n-2.}$$

Let $c_{ij}^l \left({1 \leqslant i < j \leqslant n}, 1 \leqslant l \leqslant n \right)$ be the structure constants of $\G$. Then the Lie brackets of $\G$ are given by
	$$\left[X_i, X_j \right] = \sum\limits_{l = 1}^{n-1} c_{ij}^l{X_l}; \, 1 \leqslant i < j \leqslant n.$$
In view of the Proposition \ref{prop2.8}, if $\G$ is an $MD$-algebra then $\G^2$ is commutative and we get
$$\left[X_i, X_j \right] = \sum\limits_{l = 1}^{n-1} c_{ij}^l{X_l} = 0 \Leftrightarrow
c_{ij}^l = 0; \, 1 \le i < j \leqslant k, \, 1 \leqslant l \leqslant n-1.$$ 
In order to prove Theorem \ref{thm4.3}, we need some lemmas.

\begin{lem}\label{lem4.13}
	If $\G$ is an $MDn$-algebra ($n \geqslant 3$) whose the first derived ideal $\G^1$ is commutative and 1-codimensional then $\dim \Omega_F \in \{ 0, 2 \}$, for every $F \in \G^*$.
\end{lem}

\begin{proof} Because $\G^1 = Span \left(X_1, X_2, \ldots , X_{n-1}\right)$ is $(n-1)$-dimensional commutative and $\G$ is non-commutative, so $c_{ij}^l = 0;\, 1 \leqslant i < j \leqslant n-1,\, 1 \leqslant l \leqslant n-1$ and the adjoint $ad_{X_n}$ is an isomorphism on $\G^1$. Therefore the matrix $\left(- c_{jn}^i \right)_{i, j = \overline{1, n-1}}$ of $ad_{X_n}$in the basis $\left(X_1, X_2, \ldots, X_{n-1} \right)$ of $\G^1$ is invertible. In particular, the structure constants $c_{jn}^{n-1} \, \left(1 \leqslant j \leqslant n-1 \right)$ are not concomitantly vanish. Choose $F = X_{n-1}^* \in \G^*$. It is easily seen that the matrix of the Kirillov form ${B_F}$ in the basis $(X_1, X_2, \ldots, X_n)$ as follows
	\[
		B_F =  \begin{bmatrix}
		            0 & 0 & \cdots & 0 & c_{1n}^{n-1} \\
					0 & 0 & \cdots & 0 & c_{2n}^{n-1} \\
					\vdots & \vdots & \ddots & \vdots & \vdots \\
					0 & 0 & \cdots & 0 & c_{n-1,n}^{n-1} \\
					-c_{1n}^{n-1} & -c_{2n}^{n-1} & \cdots & -c_{n-1,n}^{n-1} & 0
					\end{bmatrix}.
	\]
It is obvious that\, ${\rm rank} B_F = 2$\, because\, $c_{1n}^{n-1},\,\ldots,\, c_{n-1,n}^{n-1}$\, are not concomitantly vanish. Since \,$\G$\, is an $MD$-algebra, we get \,$\dim \Omega_F = {\rm rank} B_F \in \{ 0, 2 \}$\, for any $F \in \G^*$.  
\end{proof}

\begin{lem}\label{lem4.14}
The following $(n-k-1)$-square matrix $A$ is invertible 
	\[
		A =\begin{bmatrix}
				c_{k+1,n}^{k+1}& \ldots &c_{n-1,n}^{k+1}\\
				\vdots & \ddots & \vdots \\
				c_{k+1,n}^{n-1}& \cdots &c_{n-1,n}^{n-1}
			\end{bmatrix}.
	\]
\end{lem}                

\begin{proof}
Since $\G^1 = \left[\G, \G \right]$, there exist real numbers $\alpha_{ij} \, \left(1 \leqslant i < j \leqslant n\right)$ such that 
	$$\begin{array}{*{20}{l}} X_{k+1} & = \sum\limits_{1 \leqslant i < j \leqslant n} \alpha_{ij} \left[X_i, X_j 
	\right] \\ 
	 & = \sum \limits_{j = k+1}^{n-1} \alpha_{jn}\left[ X_j, X_n \right] + \sum \limits_{j = 1}^k \alpha 
	_{jn}\left[ X_j, {X_n} \right] + \sum \limits_{1 \leqslant i < j \leqslant n-1} \alpha_{ij} \left[ X_i, X_j 
	\right]\\
	 & = \sum \limits_{j = k+1}^{n-1} \alpha_{jn}\left[ X_j, X_n \right] + LC_1\left(\G^2 \right)\\
	& = \sum \limits_{j = k+1}^{n-1} \alpha_{jn} {\left(\sum \limits_{l = 1}^k c_{jn}^l X_l + 
	 \sum \limits_{l = k + 1}^{n-1} c_{jn}^l X_l \right)} + LC_1\left(\G^2 \right).\end{array}$$
Hence, we get
$$\begin{array}{*{20}{l}} X_{k+1} 
	 & = \sum \limits_{l = k+1}^{n-1} \left(\sum \limits_{j = k+1}^{n-1} c_{jn}^l \alpha_{jn} \right) X_l + 
	 LC_2\left(\G^2 \right); \end{array}$$
where  $LC_1\left(\G^2 \right)$ and $LG_2\left(\G^2 \right)$  are a linear combinatory of the vectors in the basis $\left(X_1, \ldots, X_k\right)$ of  $\G^2$. Because of the independence of chosen basis $\left(X_1, X_2, \ldots, X_n\right)$, these assertions imply that there exists one row-vector $Y_{k+1} \in \R^{n-k-1}$ such that $Y_{k+1}A = \left(1, 0, \ldots, 0\right)$. Similarly, there exist $Y_{k+2}, \ldots, Y_{n-1} \in \R^{n-k-1}$ such that 
		$$Y_{k+2}A = (0, 1, \ldots , 0), \,\ldots, \,Y_{n-1}A = (0, \ldots, 0, 1).$$
So there exist a real matrix $P$ such that $PA = I$, where $I$ is the unit matrix of $Mat_{n-k-1}\left(\R \right)$. Therefore $A$ is an invertible matrix. 
\end{proof}

The following lemma is the well-known result of Linear Algebra for any skew-symmetric real  4-square matrix and it can be easily verified by simple computation.
\begin{lem}\label{lem4.15}
	For any skew-symmetric real 4-square matrix $\left(a_{ij} \right)_{i,j = \overline{1,4}}$, its determinant is zero if and only if ${a_{12}}.{a_{34}} - {a_{13}}.{a_{24}} + {a_{14}}.{a_{23}} = 0.$ \hfill $\square$
\end{lem}

\begin{proof}[{\bf Proof of Theorem \ref{thm4.3}}] We now prove Theorem 4.3.

\noindent{\bf Proof of the part $\mathbf{(i)}$.}\, Let $\G$ be a real solvable Lie algebra of dimension $n$ whose the first derived ideal $\G^1 \cong {{\mathbb {R}}^{n - 1}}$ is 1-codimensional and commutative. Recall that, with notations as above, $\G^1 \equiv \R.X_1 \oplus \R.X_2 \oplus \cdots \oplus \R.X_{n-1}$.

Let $F$ be an arbitrary element in $\G^* \equiv \R^n$. Put 
$$\left\langle F, [X_i, X_n] \right\rangle = a_i; \, 1 \leqslant i \leqslant n-1.$$ 

Then, by simple computation, we can see that the matrix ${B_F}$ of the Kirillov form ${B_F}$ is given as follows.
		\[
			B_F =  \begin{bmatrix}
		            0 & 0 & \cdots & 0 & a_1 \\
					0 & 0 & \cdots & 0 & a_2 \\
					\vdots & \vdots & \ddots & \vdots & \vdots \\
					0 & 0 & \cdots & 0 & a_{n-1} \\
					-a_1 & -a_2 & \cdots & -a_{n-1} & 0
					\end{bmatrix}.
			\]
It is clear that ${\rm rank} B_F \in \left\{ {0,2} \right\}$ and, for every $F \in \G^*$, ${\rm rank} B_F$ is not concomitantly vanish. Hence, by virtue of Proposition \ref{prop2.4}, $\G$ is an $MD(n, n-1)$-algebra.

\noindent{\bf Proof of the part $\mathbf{(ii)}$.}\, We will show that if $\G$ is an real solvable Lie algebra of dimension $n > 4$ whose the first derived ideal $\G^1$ is 1-codimensional and non-commutative then $\G$ is not be an $MD$-algebra. 

Recall that, we always choose one basis $\left(X_1, X_2, \ldots, X_n \right)$ of $\G$ such that $\G^1 = Span \left(X_1, X_2, \ldots, X_{n-1} \right)$ and $\G^2 = Span \left(X_1, \ldots, X_k \right), \, k \leqslant n-2$. There are some cases which contradict each other for the values of $k$ as follows.

\noindent{\bf The first case: $k = n-2$.} \, Then, $\dim \G^2 = \dim \G^1 - 1$. According to Proposition \ref{prop2.10}, $\G$ is not an $MD$-algebra.

\noindent{\bf The second case: $k \leqslant n-3$.} \, It is sufficient to prove for just $k = n - 3$ because the proof for each $k$ in this case is similar. That means $\G^2 = Span \left(X_1, X_2, \ldots, X_{n-3} \right)$. 
	$$[X_{n-1}, X_n] = \sum\limits_{l = 1}^{n-1} c_{n-1,n}^l{X_l}; \, [X_{n-2}, X_n] = \sum\limits_{l = 1}^{n-1} c_{n-2,n}^l{X_l};$$
	$$[X_{n-2}, X_{n-1}] = \sum\limits_{l = 1}^{n-3} c_{n-2,n-1}^l{X_l};$$ 
	$$[X_i, X_n] = \sum\limits_{l = 1}^{n-3} c_{in}^l{X_l}; \, [X_i, X_j] = \sum\limits_{l = 1}^{n-3} c_{ij}^l{X_l}; \, 
1 \leqslant  i < j \leqslant n-3.$$
According to the Lemma \ref{lem4.14}, matrix $P = \begin{bmatrix}c_{n-2,n}^{n-2} & c_{n-1,n}^{n-2} \\ C_{n-2,n}^{n-1} & C_{n-1,n}^{n-1}\end{bmatrix}$ is invertible. Let $F$ be an arbitrary element of $\G^*$. Put 
	$$\left\langle F, [X_{n-2}, X_{n-1}] \right\rangle = a; \,\, \left\langle F, [X_{n-2}, X_n] \right\rangle = b; \,\, \left\langle F, [X_{n-1}, X_n] \right\rangle = c.$$
Then the matrix of the Kirillov form ${B_F}$ in the basis $\left(X_1, X_2, \ldots, X_n \right)$ is given by
$$B_F =  \begin{bmatrix}
		            0 & 0 & \cdots & 0 & * & * & * \\
					0 & 0 & \cdots & 0 & * & * & * \\
					\vdots & \vdots & \ddots & \vdots & \vdots & \vdots & \vdots \\
					0 & 0 & \cdots & 0 & * & * & * \\
					* & * & \cdots & * & 0 & a & b \\
					* & * & \cdots & * & -a & 0 & c \\
					* & * & \cdots & * & -b & -c & 0
					\end{bmatrix},$$
in which the asterisks denote the undetermined numbers. 

Let us consider the 4-square submatrices of ${B_F}$ established by the elements which are on the rows and the columns of the same ordinal numbers $i,\, n-2,\, n-1,\, n\, \,(i \leqslant n-3)$. According to Lemma \ref{lem4.13}, ${\rm rank} B_F \in \lbrace 0, 2\rbrace$ and this implies that the determinants of these considered 4-square submatrices are zero for any $F \in \G^*$. In view of Lemma \ref{lem4.15}, the following structure constants are vanished:
	$$c_{i,n-2}^l = c_{i,n-1}^l = 0; \,\, 1 \leqslant i,\, l \leqslant n-3.$$
This implies  
	$$[X_i, X_{n-2}] = [X_i, X_{n-1}] = 0, \, 1 \leqslant i \leqslant n-3.$$
Therefore, we get
	\begin{flalign*}
		Span \left(X_1, \ldots, X_{n-3} \right) & =  \G^2 = \left[\G^1, \G^1\right] \\
			& = Span \left([X_i, X_j]\,\vert \, 1 \leqslant  i,j \leqslant  n-1 \right)\\
			& = Span \left([X_{n-2}, X_{n-1}] \right).
	\end{flalign*}
So $n - 3 = \dim \G^2 \leqslant 1$, i.e. $n \leqslant 4$, which conflicts with the assumption that $n > 4$. The proof is complete.
\end{proof} 
  
\subsection{Proof of Theorem \ref{thm4.5}}

	As vector spaces (without the Lie structures), we have 
$$\G = Span \left(X_1, X_2, \ldots, X_n \right) \equiv \R^n; \,\, \G^1 = Span \left(X_1, X_2, \ldots, X_{n-1} \right) \equiv \R^{n-1}.$$
Let $A = (a_{ij})_{i,j = \overline{1,n-1}}$ be a some invertible real $(n-1)$-square matrix. 

\begin{proof}

	{\bf Proof of part $\mathbf{(i)}$.}\, We define a Lie structure on $\G$ such that $\G^1$ is commutative and $A$ is exactly the matrix of adjoint map $ad_{X_n}$ on $\G^1$ in the chosen basis $\left(X_1, X_2, \ldots, X_n\right)$. Namely, the Lie brackets $[\cdot,\cdot]_A$ on $\G$ are given as follows.
	\begin{equation}\label{equ4.1}
		[X_n, X_j]_A : = \sum \limits_{i < n} a_{ij}{X_i} \, ; \, j = 1, 2, \ldots, n-1;\,\, \text{the others are trivial}.
	\end{equation}
With such Lie structure, the derived ideal of $\G$ is commutative and exactly equals to $\G^1$. Hence, $\G$ is an $MD(n,n-1)$-algebra.

	Conversely, suppose there is a Lie structure on $\G$ whose the Lie brackets $[\cdot,\cdot]$ satisfies the above property \eqref{equ4.1}. Because the first derived ideal of $\G$ is commutative and equals to $\G^1$, one has $[X_i, X_j] = 0$ for all $i, j = 1, 2, \ldots, n-1$. On the other hand, $A$ is the matrix of adjoint map $ad_{X_n}$ on $\G^1$. Therefore, we get $\left[X_n, X_j\right] = \sum \limits_{i < n} {a_{ij}X_i}$ for all $j = 1, 2, \ldots, n-1$. That means $[\cdot,\cdot] \equiv [\cdot,\cdot]_A$ and the part (i) is proved.


\noindent {\bf Proof of part $\mathbf{(ii)}$.}\, Let $B = (b_{ij})_{i,j = \overline{1,n-1}}$ be an other invertible real $(n-1)$-square matrix and $[\cdot,\cdot]_B$ the Lie brackets on $\G$ which is defined by $B$. Then, we have

$[X_n, X_j]_B = \sum\limits_{i < n} {b_{ij}X_i} \, ;  j = 1, 2, \ldots,n-1 \, \, \, \text{(the other Lie brackets are trivial)}.$

\noindent$(\Longrightarrow)$ Suppose that $A$ and $B$ define two Lie structures on $\G$ which are isomorphic. We will show that there exist a real number $c \neq 0$ and an invertible real $(n-1)$-square matrix $C$ such that $cA = CBC^{-1}$.	Denote by $f: \left(\G, [\cdot,\cdot]_B\right) \to \left(\G, [\cdot,\cdot]_A \right)$ the isomorphism between two Lie structures on $\G$ defined by $B$ and $A$, respectively. This means that $f$ is a linear isomorphism and $f$ preserves the Lie brackets. Let $M = (c_{ij})_{i,j = \overline{1,n}}$ be the invertible $n$-square matrix of $f$ in the basis $(X_1, X_2, \ldots , X_n)$, i.e. $f(X_j) = \sum\limits_{i = 1}^{n} {{c_{ij}}{X_i}}$ for al $j = 1, 2, \ldots, n$. Since $f$ is an isomorphism, $f \left(\G^1,[\cdot,\cdot]_B\right) = \left(\G^1, [\cdot,\cdot]_A\right)$, i.e. $c_{nj} = 0$ for all $j = 1, 2, \ldots, n-1$. That means $f(X_j) = \sum \limits_{i < n} {{c_{ij}}{X_i}}$ for all $j = 1, 2, \ldots,n-1$. Put $C = (c_{ij})_{i,j = \overline{1,n-1}}$ and $c_{nn} = c$, we get $f(X_n) = \sum\limits_{i < n} {{c_{in}}{X_i} + c{X_n}}$. Then the matrix of $f$ in the basis $(X_1, X_2, \ldots, X_n)$ is given by

\centerline{$M = \begin{bmatrix}
				c_{11} & \cdots & c_{1,n-1} & c_{1n}\\
				\vdots & \cdots & \vdots & \vdots \\
				c_{n-1,1} & \cdots & c_{n-1,n - 1} & c_{n-1,n} \\
				0& \cdots &0&c
				\end{bmatrix} = \begin{bmatrix} C & * \\ 0 & c \end{bmatrix},$}
				
\noindent where the asterisk denotes the column vector $(c_{1n}, c_{2n}, \ldots, c_{n,n-1})^T$. Since $f$ is an isomorphism, $0 \neq \det M = c\det C$. Therefore, $\det C \neq 0$ and $C$ is an invertible $(n-1)$-square matrix. Furthermore, we have 
	\[
		\begin{array}{l l}
			&  f([X_n, X_j]_B) = [f(X_n), f(X_j)]_A;  \quad \forall j = 1, \ldots, n \\
			\Leftrightarrow  & f\left( \sum\limits_{i < n} {{b_{ij}}{X_i}}\right) = \left[\sum\limits_{k < n} {{c_{kn}}{X_k}}  + 
  						   c{X_n} + c_{ni}X_n, \sum\limits_{l < n} {{c_{lj}}{X_l}}\right]_A; \quad \forall j = 1, \ldots , n \\
              \Leftrightarrow & \sum \limits_{i < n} {{b_{ij}}f({X_i})}   = c \sum \limits_{l < n} {{c_{ln}}}[X_n, 
              				X_l]_A; \quad \forall j = 1, \ldots, n \\
              \Leftrightarrow & \sum\limits_{i < n} {{b_{ij}}\left( {\sum\limits_{k < n} {{c_{ki}}{X_k}} } 
              					\right)} = c \sum\limits_{l < n} {{c_{ln}}} \left( {\sum\limits_{k < n}{{a_{kl}}{X_k}} } \right);\quad \forall j = 1, \ldots, n \\
              \Leftrightarrow & \sum\limits_{k < n} {\left( {\sum\limits_{i < n} {{c_{ki}}{b_{ij}}} } \right)}{X_k} = c \sum\limits_{k < n} {\left( {\sum\limits_{l < n} {{a_{kl}}{c_{ln}}}} \right)} {X_k};\quad \forall j = 1, \ldots, n \\
              \Leftrightarrow & \sum\limits_{k < n} {\left( {\sum\limits_{i < n} {{c_{ki}}{b_{ij}}} - c\sum\limits_{l < n} {{a_{kl}}{c_{ln}}} } \right){X_k}} = 0; \quad \forall j = 1, \ldots, n \\
              \Leftrightarrow & \sum\limits_{i < n} {{c_{ki}}{b_{ij}}} - c\sum\limits_{l < n} {{a_{kl}}{c_{ln}}} = 0; \quad \forall k, j = 1, \ldots,n-1 \\
              \Leftrightarrow & \sum\limits_{i < n} {{c_{ki}}{b_{ij}}} = c\sum\limits_{l < n} {{a_{kl}}{c_{ln}}}; \quad \forall k, j = 1, \ldots,n-1 \\
              \Leftrightarrow & (CB)_{kj} = c(AC)_{kj}; \quad \forall k, j = 1, \ldots,n-1 \\
	   		  \Leftrightarrow & CB = cAC  \\
              \Leftrightarrow & cA = CBC^{-1}.
              \end{array}
             \]
              
\noindent$(\Longleftarrow)$ Conversely, if there exist a non-zero real number $c$ and an invertible $(n-1)$-square matrix $C =(c_{ij})_{i,j = \overline{1,n-1}}$ such that $cA = CBC^{-1}$. We will show that $A, B$ define on $\G$ two Lie structures which are isomorphic. Indeed, we denote by $f: \G \to \G$ the linear isomorphism of $\G$ which is defined, in the basis $(X_1, X_2, \ldots, X_n)$, by the following invertible $n$-square matrix 
	\[
		M' = \begin{bmatrix}
					c_{11} & \cdots & c_{1,n-1} & 0\\
					\vdots & \cdots & \vdots & \vdots \\
					c_{n-1,1} & \cdots & c_{n-1,n-1} & 0\\
					0 & \cdots & 0 & c
				\end{bmatrix}
		= \begin{bmatrix}
				C & 0 \\ 0 & c
			\end{bmatrix}.
	\]
It can be verified that $f$ preserves the Lie brackets of $\left(\G, [\cdot,\cdot]_B\right)$ and $\left(\G, [\cdot,\cdot]_A\right)$, i.e. $f$ is also a Lie isomorphism. The proof is complete.
\end{proof}

\section{Conclusion}

We close the paper with some remarkable comments on the problem of the classification of $MD$-class.

\subsection{$MD^{2k}$-class}

	We emphasize that the problem of classification of $MD$-algebras is still open up to now. There are at least three ways of proceeding in the classification of $MD$-class as follows:
\begin{itemize}
	\item {\bf The first way}: By fixity of the dimension of $MD$-algebras. 
	\item {\bf The second way}: By fixity of the maximal dimension of coadjoint orbits.
	\item {\bf The third way}: Combining form of the above ways.
\end{itemize} 

For example, the classifications of $MD4$-class, $MD5$-class and $MD(*, 1)$-class and $MD(*, * - 1)$-class in the paper are the results belonging to the first way, but the Arnal's list is the result of the second way. To classify $MD$-class by the second or third way, we give the following definitions.

\begin{defn}
	Each $n$-dimensional solvable Lie group $G$ whose coadjoint orbits have dimension zero or $2k \, (0 < 2k < n)$ is called \emph{an $MD^{2k}n$-group}. The Lie algebra $\G = {\rm Lie} (G)$ of $G$ is called \emph{an $MD^{2k}n$-algebra}. When we do not pay attention to the dimension of the considered group or algebra, we will call $G$ or $\G$ \emph{an $MD^{2k}$-group} or \emph{$MD^{2k}$-algebra}, respectively.
\end{defn}

\begin{defn}
	The set of all $MD^{2k}n$-algebras or $MD^{2k}$-algebras will be denoted by \emph{$MD^{2k}n$-class} or \emph{$MD^{2k}$-class}, respectively.
\end{defn}

\subsection{Examples of $MD^4$-algebras}

It has long been known that the Lie algebra ${\rm aff} (\mathbb{C})$ of the group ${\rm Aff} (\mathbb{C})$ of the affine transformations of complex straight line is an $MD^44$-algebra and the 5-dimensional Heisenberg Lie algebra $\mathfrak{h}_5$ is an $MD^45$-algebra. Now we introduce two examples of indecomposable $MD^4$-algebras.

\noindent {\bf The First Example:}\, Let $\G_{2m} = Span \left(X_1, X_2, \ldots , X_{2m} \right)$ be the $2m$-dimensional real Lie algebra ($2 \leqslant m \in \mathbb{N}$) with Lie brackets as follows

$$[X_1, X_k]: = X_k; \, [X_2, X_{2j - 1}]: = X_{2j}; \, [X_2, X_{2j}]: = - X_{2j - 1};$$

\noindent with $3 \leqslant k \leqslant 2m, \, 2 \leqslant j \leqslant m.$

\noindent {\bf The Second Example:}\, Let $\G_{2m + 1} = Span \left(X_1, X_2, \ldots , X_{2m + 1} \right)$ be the $(2m + 1)$-dimensional real Lie algebra ($2 \leqslant m \in \mathbb{N}$) with Lie brackets as follows
	\[
		[X_1, X_k]: = X_k; \, 3 \leqslant k \leqslant 2m; \, [X_3, X_4] = X_{2m + 1}; \, [X_1, X_{2m + 1}] = 2X_{2m + 1};
	\]
	\[
		[X_2, X_{2j}]: = - X_{2j - 1}; \, [X_2, X_{2j - 1}]: = X_{2j}; \, 2 \leqslant j \leqslant m.
	\]

Upon simple computation, taking Proposition \ref{prop2.4} into account we get the following proposition.

\begin{prop}
	$\G_{2m}$ is an indecomposable $2m$-dimensional $MD^4$-algebra and $\G_{2m + 1}$ is an indecomposable $(2m + 1)$-dimensional $MD^4$-algebra.
\end{prop}

\subsection{Some open problems} We have at least two open problems as follows.

\subsubsection{\bf Classify $MD(n, m)$-class, $MD(n, n - m)$-class; $2 \leqslant m \leqslant n - 2, n \geqslant 6$}

\subsubsection{\bf Classify $MD^{2k}n$-class or $MD^{2k}$-class; $2k \geqslant 4, n \geqslant 6$} \hfill

\bigskip
	In the next papers, we will discuss the classification of $MD(n, 2)$-class, $MD(n, n - 2)$-class and $MD^4n$-class for $6 \leqslant n \leqslant 8$. 

\begin{aknow} 
	The authors would like to take this opportunity to thank the scientific program of University of Economics and Law, Vietnam National University -- Ho Chi Minh City for financial supports.
\end{aknow}

\bibliographystyle{amsxport}

\end{document}